\documentclass[11pt]{amsart}

\usepackage{graphicx}
\usepackage{amssymb}
\usepackage{amsmath}
\usepackage{enumerate}
\usepackage{amsfonts}
\usepackage{amsthm}
\usepackage[all,2cell,import]{xy} \UseAllTwocells
\usepackage[colorlinks=true, pdfstartview=FitV, linkcolor=black, 
			   citecolor=black, urlcolor=black]{hyperref}
\usepackage{bookmark}
\theoremstyle{definition}
\newtheorem{theorem}{Theorem}
\newtheorem{sectional}{Theorem}[section]

\newtheorem{prop}[sectional]{Proposition}
\newtheorem{lemma}[sectional]{Lemma}

\newtheorem{conjecture}[]{Conjecture}

\newtheorem{defn}[sectional]{Definition}
\newtheorem{notation}[sectional]{Notation}
\newtheorem{warning}[sectional]{Warning}
\newtheorem{example}[sectional]{Example}
\newtheorem{construction}[sectional]{Construction}

\newtheorem{remark}[sectional]{Remark}
\newtheorem{image}[sectional]{Figure} 

\usepackage{chngcntr}
\counterwithin{equation}{theorem}

\newcommand{\nc}{\newcommand}
\nc{\DMO}{\DeclareMathOperator}	

\nc{\newnotation}{\nomenclature}
\nc{\wrap}{\cW}
\nc{\Cob}{\mathsf{Cob}}
\nc{\mul}{\mathsf{Mul}}
\nc{\fat}{\mathsf{fat}}
\nc{\cob}{\mathsf{Cob}}
\nc{\coh}{\mathsf{Coh}}
\nc{\core}{\mathsf{core}}
\nc{\idem}{\mathsf{Idem}}
\nc{\sets}{\mathsf{Sets}}
\nc{\near}{\mathsf{near}}
\nc{\sing}{\mathsf{Sing}}
\nc{\symp}{\mathsf{Symp}}
\nc{\perf}{\mathsf{Perf}}
\nc{\ssets}{\mathsf{sSets}}
\nc{\cmpct}{\mathsf{cmpct}}
\nc{\finite}{\mathsf{Finite}}
\nc{\compact}{\mathsf{cmpct}}
\nc{\pwrap}{\mathsf{PWrap}}
\nc{\coder}{\mathsf{Coder}}
\nc{\bimod}{\mathsf{Bimod}}
\nc{\grmod}{\mathsf{GrMod}}
\nc{\spaces}{\mathsf{Spaces}}
\nc{\pwrms}{\mathsf{PWrFuk}_{M,S}}
\nc{\pwrmf}{\mathsf{PWrFuk}_{M,F}}
\nc{\pwrapmf}{\mathsf{PWrFuk}_{M,F}}
\nc{\fuk}{\mathsf{Fukaya}}
\nc{\infwr}{\mathsf{InfWr}}
\nc{\fukaya}{\mathsf{Fukaya}}
\nc{\autml}{\mathsf{Aut}_{M,\Lambda}}
\nc{\fukml}{\mathsf{Fukaya}_{M,\Lambda}}
\nc{\fukmle}{\mathsf{Fukaya}_{M,\Lambda,\epsilon}}
\nc{\fukmod}{\wrfukcompact(M)\modules}
\nc{\lag}{\mathsf{Lag}}
\nc{\Lag}{\mathsf{Lag}}
\nc{\lagm}{\lag_M}
\nc{\lago}{\lag^o}
\nc{\lagml}{\lag_{M,\Lambda}} % For when I get lazy.
\nc{\lagmle}{\lag_{M,\Lambda,\epsilon}}
\nc{\fun}{\mathsf{Fun}}
\nc{\vect}{\mathsf{Vect}}
\nc{\chain}{\mathsf{Chain}}
\nc{\wrfuk}{\mathsf{WrFukaya}}
\nc{\wrfukcompact}{\mathsf{WrFukaya}_{\mathsf{cmpct}}}
\nc{\pwrfuk}{\mathsf{PWrFukaya}}
\nc{\inffuk}{\mathsf{InfFuk}}
\nc{\pwrfukml}{\mathsf{PWrFukaya}_{M,\Lambda}}
\nc{\inffukml}{\mathsf{InfFuk}_{M,\Lambda}}
\nc{\nattrans}{\mathsf{NatTrans}}
\nc{\corres}{\mathsf{Corres}}
\nc{\fukep}{\fukaya_\Lambda(M,\epsilon)}
\nc{\fukepop}{\fukaya_\Lambda(M,\epsilon)^{\op}}
\nc{\lagep}{\lag_\Lambda(M,\epsilon)}
\DMO{\cyl}{cyl} % Cylindrical
\nc{\dbcoh}{D^b\mathsf{Coh}}
\nc{\corr}{\mathsf{Corr}}

\nc{\cat}{\mathsf{Cat}}
\nc{\Cat}{\mathsf{Cat}}
\nc{\ainfty}{\mathsf{A}_\infty}
\nc{\inftycat}{\mathcal{C}\!\operatorname{at}_\infty}
\nc{\Ainftycat}{\mathcal{C}\!\operatorname{at}_{A_\infty}}
\nc{\ainftycat}{\mathcal{C}\!\operatorname{at}_{A_\infty}}
\nc{\stablecat}{\mathcal{C}\!\operatorname{at}_\infty^{\Ex}}

\DMO{\im}{im}
\DMO{\ev}{ev}
\DMO{\inj}{inj}
\DMO{\fib}{fib}
\DMO{\conf}{Conf}
\DMO{\chains}{Chains}
\DMO{\cochains}{Cochains}
\DMO{\cone}{Cone}
\DMO{\ran}{Ran}
\DMO{\rot}{Rot}
\DMO{\leg}{Leg}
\DMO{\imm}{imm}
\DMO{\adj}{adj}
\DMO{\Crit}{Crit}
\DMO{\tree}{Tree}
\DMO{\cube}{Cube}
\DMO{\deep}{deep}
\DMO{\back}{back}
\DMO{\front}{front}
\DMO{\flow}{Flow}
\DMO{\floer}{Floer}
\DMO{\maps}{Maps}
\DMO{\exact}{exact}
\DMO{\excess}{Excess}
\DMO{\Decomp}{Decomp}
\DMO{\decomp}{Decomp}
\DMO{\collar}{collar}
\DMO{\yoneda}{Yoneda}
\DMO{\hamspace}{Ham}
\DMO{\sympspace}{Symp}
\DMO{\holomaps}{Holomaps}
\DMO{\comp}{Comp}
\DMO{\crit}{Crit}
\DMO{\test}{{test}}
\DMO{\sign}{sign}
\DMO{\topp}{top}
\DMO{\indx}{Index}
\DMO{\Break}{Break} % Partitions
\DMO{\zero}{zero} %Zero
\DMO{\ob}{Ob}
\DMO{\gr}{Gr} % Grassmanian
\DMO{\Gr}{Gr} % Grassmanian
\DMO{\cl}{Cl} % Clifford Algebra
\DMO{\grlag}{GrLag}
\DMO{\GrLag}{GrLag}
\DMO{\Pin}{Pin}
\DMO{\Graph}{Graph}
\DMO{\grph}{Graph}
\DMO{\pin}{Pin}
\DMO{\gap}{Gap}
\DMO{\Ex}{Ex}
\DMO{\id}{id}
\DMO{\End}{End}
\DMO{\sym}{Sym} 
\DMO{\aut}{Aut}
\DMO{\DK}{DK} %Dold-Kan
\DMO{\poly}{poly} % Polynomial deRham forms
\DMO{\diff}{Diff}
\DMO{\coll}{coll}
\DMO{\dist}{dist} %Distance function
\DMO{\coker}{coker} %Cokernel
\nc{\kernel}{\ker} %Kernel
\DMO{\sspan}{span}
\DMO{\hocolim}{hocolim}	
\DMO{\holim}{holim}
\DMO{\sk}{sk}

\DMO{\ho}{ho}
\DMO{\fin}{fin}
\DMO{\tor}{Tor}
\DMO{\ext}{Ext}
\DMO{\ret}{Ret}
\DMO{\ham}{Ham}
\DMO{\con}{con}
\DMO{\leaf}{leaf}
\DMO{\supp}{supp}
\DMO{\edge}{edge}
\DMO{\colim}{colim}
\DMO{\edges}{edges}
\DMO{\Image}{image}
\DMO{\roots}{roots}
\DMO{\height}{height}
\DMO{\finmod}{FinMod}
\DMO{\leaves}{leaves}
\DMO{\planar}{planar}
\DMO{\vertices}{vertices}

\nc{\lagg}{\lag^{\cG}}
\nc{\iso}{\mathsf{Iso}}
\nc{\Set}{\mathsf{Set}}
\nc{\ass}{\mathsf{ \bf Ass}}
\nc{\Mod}{\mathsf{Mod}}
\nc{\modules}{\mathsf{Mod}}
\nc{\sset}{\mathsf{sSet}}
\nc{\liou}{\mathsf{Liou}}
\nc{\poset}{\mathsf{Poset}}
\nc{\trno}{T^*\RR^n_{\geq 0}}
\nc{\spectra}{\mathsf{Spectra}}
\nc{\tensorfin}{\tensor^{\fin}}
\nc{\lagptg}{\lag_{pt,pt}^{\cG}}
\nc{\Fin}{\mathcal{F}\mathsf{in}}
\nc{\lagnl}{\lag_{N,\Lambda}}
\nc{\lagmlg}{\lag_{M,\Lambda}^{\cG}}
\nc{\lagsplit}{\lag^{\mathsf{split}}}
\nc{\lagktimes}{(\lag^{\dd k})^\times}
\nc{\lagplanar}{\lag^{\times,\planar}}

\nc{\smsh}{\wedge}
\nc{\un}{\underline}
\nc{\xto}{\xrightarrow}
\nc{\xra}{\xto}
\nc{\tensor}{\otimes}
\nc{\del}{\partial}
\nc{\dd}{\diamond}
\nc{\tri}{\triangle}
\nc{\bb}{\Box}
\nc{\into}{\hookrightarrow}
\nc{\onto}{\twoheadrightarrow}
\nc{\contains}{\supset}
\nc{\transverse}{\pitchfork}
\nc{\uncirc}{\underline{\circ}}
\nc{\Jbar}{\overline{J}}
\nc{\Fbar}{\overline{F}}
\nc{\delbar}{\overline{\del}}
\nc{\thetabar}{\overline{\theta}}
\nc{\omegabar}{\overline{\omega}}

\nc{\colldiff}{\diff^{\del}} 
\nc{\trbar}{\overline{T^*\RR}}
\nc{\tr}{T^*\RR}
\nc{\tsa}{Ts\cA}
\nc{\tsb}{Ts\cB}
\nc{\cmbar}{\overline{\cM}}
\nc{\crbar}{\overline{\cR}}
\nc{\fcrit}{\ff^{crit}}
\nc{\fsubcrit}{\ff^{subcrit}}

\nc{\vece}{ {\vec \epsilon}}	
\nc{\vecd}{ {\vec \delta}}
\nc{\ov}{\overline}
\DMO{\op}{op}
\nc{\opp}{ ^{\op}}

\nc{\hiro}{\textcolor{blue}}

\nc{\eqn}{\begin{equation}}
\nc{\eqnn}{\begin{equation}\nonumber}
\nc{\eqnd}{\end{equation}}
\nc{\enum}{\begin{enumerate}}
\nc{\enumd}{\end{enumerate}}

\def\cA{\mathcal A}\def\cB{\mathcal B}\def\cC{\mathcal C}
\def\cG{\mathcal G}
\def\cL{\mathcal L}
\def\cM{\mathcal M}
\def\cR{\mathcal R}
\def\cW{\mathcal W}

\def\CC{\mathbb C}

\def\LL{\mathbb L}

\def\RR{\mathbb R}\def\SS{\mathbb S}
\def\XX{\mathbb X}
\def\ZZ{\mathbb Z}

\def\ff{\mathfrak f}
\def\fg{\mathfrak g}

\hypersetup{pageanchor=false}

	\title{Generation for Lagrangian cobordisms in Weinstein manifolds}
	\author{Hiro Lee Tanaka}
\begin{document}

\begin{abstract}
We prove that Lagrangian cocores and Lagrangian linking disks of a stopped Weinstein manifold generate the Lagrangian cobordism $\infty$-category. 

As a geometric consequence, we see that any brane (after stabilization) admits a Lagrangian cobordism to a disjoint union of some standard collection of branes (cocores, linking disks, and a zero object). For example, when our stopped Weinstein manifold is a point stopped by itself, we find that any exact brane in Euclidean space admits a Lagrangian cobordism to a disjoint union of cotangent fibers and a zero object. (This is a stronger statement than one could obtain from purely Fukaya-categorical generation results.) Our methods are constructive. For example, when our Weinstein manifold is a point, after stabilization we can resolve the conormal to a compact manifold $A \subset \RR^n$ by a sequence of cotangent fibers; the resulting filtration realizes, after passage to the wrapped Fukaya category, the Morse cochain complex of $A$ associated to (and hence filtered by) a generic ``distance to a point'' function; the associated gradeds are the reduced homologies of the Morse attaching spheres.

There is also an algebraic consequence. Lagrangian cobordism theory is conjectured (in analogue to classical cobordism theory) to be linear over a ring spectrum $\cL$ controlling Lagrangian cobordisms between cotangent fibers in Euclidean spaces. Our main theorem gives strong evidence for this conjecture: The $\infty$-category of Lagrangians and their cobordisms in $\RR^\infty$ is equivalent to a full subcategory of modules over $\cL$.

We conclude by proving a $\pi_0$-level theorem that gives further evidence of the above conjecture: We exhibit a $\pi_0$-level symmetric monoidal structure compatible with the linear structure of $\cL$-modules.
\end{abstract}

	\maketitle

\section{Introduction}
Let $M$ be Liouville manifold---that is, an exact symplectic manifold whose Liouville flow induces a conical end outside some compact set. Fix also some subset $\Lambda \subset M$.

To this data, one can associate an $\infty$-category $\lag_\Lambda(M)$ whose objects are certain exact Lagrangian branes in $M \times T^*\RR^N$ for $N \geq 0$, and whose morphisms are Lagrangian cobordisms satisfying a non-characteristic condition with respect to $\Lambda$. (See Definition~\ref{defn. non-characteristic cobordism}.)

In this paper, we consider the case when $M$ is Weinstein and $\Lambda$ arises by choosing stops.

\begin{theorem}\label{thm:main-vague}
The collection of Lagrangian cocores and linking disks generate $\lag_\Lambda(M)$ as a stable $\infty$-category.
\end{theorem}

We refer the reader to Theorem~\ref{theorem. cocores} for a more precise statement. 

The above is analagous to generation results for partially wrapped Fukaya categories~\cite{CDGG, gps2} and is in fact inspired by ideas in~\cite{gps2}. However, Fukaya-categorical generation results sometimes leave us without the ability to {\em geometrically} construct or realize an arbitrary brane as ``resolved'' by the generating objects.
As it turns out, because we work in the setting of cobordisms, the proof method yields the following geometric result:

\begin{theorem}\label{thm:geometric-version}
Fix $X$ an exact brane inside $M$. Then there exists an exact Lagrangian cobordism whose ends are given by a stabilization of $X$, a zero object, and stabilizations of Lagrangian cocores and linking disks.
\end{theorem}

\begin{example}
When $M=\Lambda=pt$, the result shows that for any exact brane $X \subset T^*\RR^N$ there exists an exact Lagrangian cobordism one of whose ends is (a possibly stabilized) copy of $X$, another is a zero object, while the other ends are cotangent fibers.
\end{example}

\begin{remark}
Let us be precise by what we mean by a Lagrangian cobordism in Theorem~\ref{thm:geometric-version}. We mean that there exists an integer $N \geq 0$ and an exact Lagrangian submanifold
	\eqnn
	W \subset M \times T^*\RR^N \times T^*\RR_t
	\eqnd
whose projection to $T^*\RR_t$ is equal---outside of some compact $K \subset T^*\RR_t$---to a finite union of horizontal rays (i.e., contained in the zero section) and of upward-pointing vertical rays (i.e., contained in some collection of cotangent fibers and in the upper-half plane of $T^*\RR_t$). Along the positive ray in the zero section, $W$ is equal to a direct product of this positive ray with $X \times T^*_0 \RR^N$, i.e., $W$ is collared by a stabilization of $X$.

On exactly one other ray, $W$ is equal to a direct product of the ray with a zero object of $\lag_\Lambda(M)$; in fact, this zero object may be taken to be a copy of $X$ times a curve in $T^*\RR_t$ avoiding the zero section of $T^*\RR_t$. 

Over all other rays, $W$ is equal to a direct product of the ray with a stabilization of a cocore or a linking disk.

We emphasize that under our rules of engagement (i.e., by our definition of morphism in $\lag_\Lambda(M)$), this cobordism is {\em not} in general equivalent to a cobordism collared by only two rays (both contained in the zero section); that is, it cannot be interpreted as a {\em morphism} with domain the disjoint union of cocores and disks, and codomain a stabilization of $X$. 
\end{remark}

\begin{remark}
The reader may be familiar with the fact that classical cobordism theory is often trivial when one considers non-compact cobordisms between non-compact manifolds. In contrast, our cobordisms have strong controls both on their Lagrangian structures (they must be exactly embedded) and their behavior at infinity (their primitives must vanish near infinity) so the existence of certain Lagrangian cobordisms is not a triviality. This is also reflected in the algebraic content of the Lagrangian cobordisms we produce, which we will explain shortly.
\end{remark}

\begin{remark}
Though the two above theorems have Floer-theoretical consequences, they are proved independently of any Floer-theoretical techniques. In particular, both the statements and proofs of the theorems require no mention of Fukaya categories.
\end{remark}

We have stated that $\lag_\Lambda(M)$ is a {\em stable} $\infty$-category~\cite{nadler-tanaka}. Informally, this means that $\lag_\Lambda(M)$ behaves like a pretriangulated dg-category, or pretriangulated $A_\infty$-category, but enriched over spectra (not chain complexes) in the sense of stable homotopy theory. For us, the import will be that $\lag_\Lambda(M)$ has a notion of mapping cone sequences.

We set some notation to make a precise statement. 

\begin{notation}
One can choose different brane structures to decorate Lagrangians, and each choice of brane convention results in a different $\infty$-category of Lagrangian cobordisms. In this work, we demand that every (Lagrangian) brane in sight is equipped with a primitive which vanishes near infinity. (In particular, every brane is conical at infinity.) When relevant, we denote a choice of brane convention by $\cB$ and we indicate the dependence by $\lag^\cB_\Lambda(M)$. A typical choice of $\cB$ is to demand that every Lagrangian in sight is equipped with a primitive vanishing at infinity, a relative Pin structure, and a grading.  When the choice of $\cB$ is immaterial, we will write $\lag_\Lambda(M)$.
\end{notation}

\begin{notation}\label{notation:core}
Also recall that if $Z$ is the Liouville vector field of $M$ (pointing outward along the conical end of $M$), one can define the {\em core}, or {\em skeleton} of $M$ to be the set
	\eqnn
	\core(M) := \bigcap_{t >0} \flow^Z_{-t}(M^o)
	\eqnd
where $M^o$ is the (compact) complement of a conical end. Equivalently, $\core(M)$ is the set of points that do not escape $M^o$ under the flow of $Z$. 
\end{notation}

Theorem~\ref{thm:main-vague} may be stated more precisely as follows:

\begin{theorem}\label{theorem. cocores}
Assume $Z$ is gradient-like for a Morse function. 
\enum
\item Let $\{D\}$ denote the collection of Lagrangian cocores, and let $\{D^\alpha\}$ denote the collection of Lagrangian cocores equipped with brane structures. Then for $\Lambda = \core(M)$, the collection $\{D^\alpha\}$ generates $\lag_\Lambda(M)$ as a stable $\infty$-category. 
\item More generally, choose a stop $\ff \subset \del_\infty M$. Let $\{D\}$ denote the collection of Lagrangian cocores, and of Lagrangian linking disks of the critical components of $\ff$. Then for $\Lambda = \core(M) \bigcup \cup_t \flow_t^Z(\ff)$, the collection $\{D^\alpha\}$ generates $\lag_\Lambda(M)$ as a stable $\infty$-category.
\enumd
\end{theorem}

That is, given any object $X \in \lag_\Lambda(M)$,  $X$ can be written as an iterated extension of the $D^{\alpha}$. Put another way, one can find a finite sequence of objects and maps between them
	\eqnn
	0 \simeq X_n \to X_{n-1} \to \ldots \to X_0 \simeq X
	\eqnd
in $\lag_\Lambda(M)$ such that for all $i$, the mapping cone $X_i/X_{i-1}$ is equivalent to one of the $D^{\alpha}$; or, equivalently, the fiber of the map $X_i \to X_{i-1}$ is equivalent to (a shift of) one of the $D^\alpha$. See also Figure~\ref{figure:filtration-staircase}.

\begin{remark}
	Because $\lag_\Lambda(M)$ in general is not known to be idempotent-closed, the generation statement is stronger than split-generation. 
	\end{remark}

\begin{remark}\label{remark:resolution-by-cobordisms}
The proof of Theorem~\ref{theorem. cocores} is constructive.  Concretely, given any object $X$ of the wrapped category of $M$, our proof {exhibits} a Lagrangian cobordism from $X \times \RR^N$ (i.e., a stabilized version of $X$) to a disjoint union of multiple branes, each of which is either a Lagrangian cocore or a linking disk (possibly stabilized).
\end{remark}

The theory of stable $\infty$-categories is merely analogous to that of pretriangulated categories (and not equivalent). Informally, this is because a stable $\infty$-category is a priori linear not over a ring in the usual sense, but over a ring spectrum in the sense of stable homotopy theory. Accordingly, to study Lagrangian cobordisms, we would like to identify what this base ring spectrum {is}.
Theorem~\ref{theorem. cocores} enables us to make significant headway.

Recall that every stable $\infty$-category has a shift operation $L \mapsto L[1]$, where $L[1]$ is defined as the cofiber of the zero map $L \to 0$.  Given the point $M=\Lambda=pt$,  let $L = pt$ be the unique non-empty (exact) Lagrangian submanifold. We fix a choice of brane convention $\cB$ for which the set of brane structures on $L=pt$ is acted upon transitively by the shift operation. Then Theorem~\ref{theorem. cocores} shows that $\Lag^\cB_{pt}(pt)$ is generated by a single object: the point. 

\begin{notation}\label{notation:cL}
We let $\cL^{\cB} = \End(pt)$ denote the ring of endomorphisms. Note that because $\lag^{\cB}_{pt}(pt)$ is stable, we may consider $\cL^{\cB}$ as an $A_\infty$ ring spectrum.
\end{notation}

The following  is immediate from the definition of generation and the fact that $\lag$ is a stable $\infty$-category:

\begin{theorem}\label{theorem. point}
		$\Lag_{pt}^\cB(pt)$ is equivalent to a full subcategory of modules over $\cL^{\cB}$, consisting of those modules that can be presented as finite extensions of shifts of $\cL^{\cB}$.
\end{theorem}
		
\begin{remark}
By modules over $\cL^\cB$, we of course mean modules of spectra.
\end{remark}

The partially wrapped Fukaya categories of exact branes are all linear over the integers---these integers are concretely found as endomorphisms of a point in $M=\Lambda=pt$. The naive ``proof'' of this fact is in the definitions, as all wrapped complexes are defined $\ZZ$-linearly in the exact setting. In contrast, our definition of $\lag_\Lambda(M)$ is not constructed linearly over $\cL^\cB$, so one must find a geometric reason why arbitrary choices of $M$ and $\Lambda$ still yield $\cL^\cB$-linear $\infty$-categories.

To do so, let us recall the following fact from algebra: If one wants to exhibit an $R$-linear structure on a $\ZZ$-linear $A_\infty$-category $\cA$, it suffices to produce an action
	\eqnn
	R\Mod^{fg} \times \cA \to \cA
	\eqnd
preserving zero objects and mapping cone sequences in each variable. (Here, $R\Mod^{fg}$ is the dg-category of finitely generated complexes of $R$-modules.)

To that end, we prove a $\pi_0$ version of a statement we hope to enrich in later work:

\begin{theorem}\label{theorem. pt action}
For any $M$ and $\Lambda$, let $\ho\Lag^{\cB}_{\Lambda}(M)$ denote the homotopy category of $\Lag_{\Lambda}^{\cB}(M)$. Then $\ho\Lag^{\cB}_{pt}(pt)$ is a symmetric monoidal category, and the direct product of branes induces an action
	\eqnn
	\ho\Lag^{\cB}_{pt}(pt) \times
	\ho\Lag^{\cB}_{\Lambda}(M) \to
	\ho\Lag^{\cB}_{\Lambda}(M).
	\eqnd
Moreover, the action preserves the zero object and exact triangles in each variable.
\end{theorem}

\begin{remark}
Because $\lag_\Lambda^{\cB}(M)$ is stable, its homotopy category is naturally triangulated, just as the cohomology category of a pretriangulated dg-category is triangulated. Concretely, exact triangles in $\ho \lag_{\Lambda}^{\cB}(M)$ are those sequences of morphisms that arise from a co/fiber sequence in $\lag_\Lambda^{\cB}(M)$. 
\end{remark}

\begin{remark}
A higher version of Theorem~\ref{theorem. pt action}---that is, a version without the prefix $\ho$---would imply that for any $M$, $\lag^{\cB}_{\Lambda}(M)$ is linear over the ring spectrum $\cL^{\cB}$. It would also imply that (because $pt$ is the unit of the symmetric monoidal structure) $\cL^{\cB}$ is an $E_\infty$ ring spectrum. Informally, as $\ZZ$ is to the exact wrapped Fukaya category, $\cL^{\cB}$ is to exact Lagrangian cobordisms.
\end{remark}

\subsection{Applications of the results}\label{section:more-motivation}
Spectral algebra is typically a richer home for invariants than homological algebra. Accordingly, it is expected that $\lag_\Lambda(M)$ is a richer  invariant than the partially wrapped Fukaya category. For example, we have the following conjecture from~\cite{nadler-tanaka}:

\begin{conjecture}\label{conjecture:main}
For any choice of Liouville $M$, of $\Lambda \subset M$, and of $\cB$, we have that $\lag^{\cB}_\Lambda(M)$ is linear over $\cL^\cB$.

Further, suppose that $\Lambda$ is obtained as the union of $\core(M)$ and the Liouville flow of a stop. Then there is an equivalence
	\eqnn
	\lag^{\cB}_{\Lambda}(M) \tensor_{\cL^\cB} H\ZZ
	\simeq
	\cW_\Lambda(M)
	\eqnd
to the partially wrapped Fukaya category of $M$ with stop $\Lambda \cap \del_\infty M$.
\end{conjecture}

\begin{remark}
Any ring can be naturally made into a ring spectrum; the notation $H\ZZ$ merely emphasizes that we think of the usual ring $\ZZ$ as a ring spectrum, and hence we may tensor with $H\ZZ$ along any map of ring spectra $\cL^\cB \to H\ZZ$.

Thus, the conjecture states that the partially wrapped category is obtained as a ``change of coefficients'' from the Lagrangian cobordism $\infty$-category. Theorem~\ref{theorem. pt action} is evidence toward the $\cL^\cB$-linearity of $\lag^\cB$.
\end{remark}

Before we proceed, we recall the following:

\begin{theorem}\label{theorem:lag-detects-fuk}
Let $\Lambda$ contain the core $\core(M)$ (Notation~\ref{notation:core}).
\enum
\item\cite{tanaka-pairing} There exists a functor
	\eqnn
	\lag_\Lambda(M) \times  \fukaya^{\cmpct}_\Lambda(M)^{\op} \to \chain_\ZZ
	\eqnd
pairing $\lag_\Lambda(M)$ with the Fukaya category of compact Lagrangians inside $M$. The target is the $\infty$-category of chain complexes over $\ZZ$. 
\item \cite{tanaka-exact} Moreover, the associated functor
	\eqnn
	\Xi: 
	\lag_\Lambda(M) \to \fun_{A_\infty}(\fukaya^{\cmpct}_\Lambda(M)^{\op}, \chain_\ZZ	)
	\eqnd
is exact, meaning it sends exact triangles in the domain to exact triangles in the (pretriangulated closure of) the Fukaya category.
\enumd
\end{theorem}

\begin{remark}
It was proven in~\cite{tanaka-surgery} that Polterovich surgery of Lagrangians produces exact sequences in $\Lag_\Lambda(M)$.
Thus, Theorem~\ref{theorem:lag-detects-fuk} gives another proof that surgery results in exact sequences in modules over the Fukaya category of compact branes. Analogues of such a fact were proven in~\cite{mak-wu} and in the 2007 (pre-publication) version of ~\cite{fooo}.
\end{remark}

\begin{example}\label{example:morse-theory-RN}
Let  $M=\Lambda=pt$, in which case an object of $\lag_\Lambda(M)$ is given by any brane in $T^*\RR^N$, for some $N$. Let $A\subset \RR^N$ be a compact manifold and let $X$ be the conormal to $A$. Then Theorem~\ref{thm:main-vague} guarantees that we can write $X$ as an iterated extension by a cotangent fiber of $T^*\RR^N$. Tracing through our proof, one sees that this iteration can be modelled geometrically as follows: 

Pick a generic point $y \in \RR^N$ and let $L$ be its cotangent fiber. Then wrapping by the geodesic flow models the Reeb chords from $\del_\infty X$ to $\del_\infty L$ as critical points of the distance functional from $A$ to $y$. (Considering such functions is a standard way to prove the existence of Morse functions for compact smooth manifolds $A$.)

Running through the proof of Theorem~\ref{theorem. cocores}, we obtain a commutative diagram in $\lag_\Lambda(M)$ as in Figure~\ref{figure:filtration-staircase}, where every square is a pushout square. 

\begin{image}
\label{figure:filtration-staircase}
\eqnn
\xymatrix{
			 0 \ar[r]  	&L[a_{n-2}] \ar[d] \ar[r] & \ldots \ar[r] \ar[d] & \ldots \ar[d]\ar[r] 	& X_{n-1} = L[a_{n-1}+1] \ar[d] \\
						& 0 \ar[r] 		& \ddots \ar[r] \ar[d]	& \ldots \ar[r] \ar[d] & \vdots \ar[d] \\
						&			& 0 \ar[r]					&		L[a_0] \ar[d] \ar[r]	& X_1 \ar[d] \\
						&			&					&		0 \ar[r] 	& X_0 = X.
}
\eqnd
\end{image}

This means, for example, that $L[a_0] \to X_1 \to X_0$ is an exact sequence.\footnote{Or, upon passage to the homotopy category, an exact triangle (in the lingo of triangulated categories).} We have indexed the $X_i$ so as to agree with the indexing following Theorem~\ref{theorem. cocores} and Proposition~\ref{prop. generation and cones}. Informally, $X_1$ is obtained by surgering $X_0$ with $L$ at a covector representing the  minimum for the distance Morse function. $X_2$ is obtained by further surgering $X_1$ along a covector representing the next critical point of the distance Morse function, and so on; the claim is that the final stage $X_{n-1}$ is obtained as an object equivalent to cotangent fiber.

For an algebraic consequence of this geometry, we may now look at the image of this diagram under the functor 
\eqnn
	\Xi: \lag_{pt}(pt) \to \chain_\ZZ
\eqnd
(see Remark~\ref{theorem:lag-detects-fuk}), where we have used the fact that the $\infty$-category of $\ZZ$-linear chain complexes is (equivalent to) the $\infty$-category of modules over the compact Fukaya category of a point. Because $L$ is sent to $\ZZ$ by~\cite{tanaka-pairing}, the above diagram exhibits the image of the conormal $X$ as an iterated extension of shifts of the chain complex $\ZZ$. Moreover, the image of the righthand column under $\Xi$ is a filtration of the chain complex $\Xi(X)$. This filtration is the Morse filtration on the chain complex $C_*(A)$, induced by the distance function from $A$ to $y$. (In particular, the shifted $\ZZ$ appearing as associated gradeds are the reduced homologies of the Morse attaching spheres.)
\end{example}

\begin{example}
We saw in Example~\ref{example:morse-theory-RN} that after applying $\Xi$, the proof method of Theorem~\ref{theorem. cocores} seems to recover a filtered Morse chain complex of submanifolds $A \subset \RR^N$.

Theorem~\ref{theorem. point} suggests further that, when restricted to the special case of $M=\Lambda=pt$ and to those branes that arise as conormal bundles of submanifolds $A$, $\lag_\Lambda(M)$ recovers information about the stable homotopy type of $A$, smashed with a coefficient ring spectrum $\cL^\cB$. This is consistent with Conjecture~\ref{conjecture:main}, as chains of $A$ can be recovered simply by observing:
	\eqnn
	\Xi(\Sigma^\infty A_+ \tensor \cL^\cB) \simeq \Sigma^\infty A_+ \tensor \cL^{\cB} \tensor_{\cL^\cB} H\ZZ \simeq \Sigma^\infty A_+ \tensor H\ZZ \simeq C_*(Z).
	\eqnd
Here, we have used $\tensor$ to mean the smash product for spectra; this is also sometimes written as $\smsh$ or as $\tensor_{\SS}$--i.e., tensoring over the sphere spectrum.
\end{example}

\subsection{Remarks on the proofs}
\label{section. remarks on proofs}
		\subsubsection{Proving generation.}
		The proof method for Theorem~\ref{theorem. cocores} is a strategy we learned (for the wrapped Fukaya setting) from conversations with Shende at the KIAS Higher Categories and Mirror Symmetry conference, where it was suggested by Shende that the same proof should work for Lagrangian cobordisms; see~\cite{gps2} for the argument in the wrapped Fukaya setting. In the present work, we need not perform any computations in the Fukaya category. As has been the trend in recent years, if the necessary lemmata can be reduced to statements with no mention of holomorphic disk computations, one expects to be able to carry out the proof in the Lagrangian cobordism setting. This is what we have done.

		The proof is as follows: By design, a brane $L$ and its stabilization $L \times T^*_{0}\RR$ are equivalent objects in the Lagrangian cobordism $\infty$-category. Via genericity, one can flow $L \times T^*_{0}\RR$ past Lagrangian cocores in $M \times T^*\RR$ in such a way that (i) we only introduce one intersection point at a time, and (ii) the result is contained entirely in $M \times T^*\RR_{\tau < 0}$ (i.e., where the cotangent component is negative). The first condition guarantees that we can resolve each intersection point via Polterovich surgery to obtain a mapping cone sequence~\cite{tanaka-surgery}, while the latter states that the end result is a zero object in the Lagrangian cobordism $\infty$-category~\cite{nadler-tanaka}. 
		
There are some details to be filled in, of course---see Section~\ref{section. proof generation}. The main new ingredient is that attaching exact handles at infinity does not change the equivalence class of an object in $\lag_\Lambda(M)$. (The corresponding statement for the partially wrapped Fukaya category was shown in~\cite{gps2}, but by using Viterbo restriction maps, which are not available to us.) The proof we give yields a vertically bounded Lagrangian cobordism realizing the equivalence; this boundedness is also necessary in exhibiting a proof of Theorem~\ref{thm:geometric-version}.

\begin{remark}[Contrast with~\cite{CDGG}.]
		The reader can find a proof of the same wrapped Fukaya category result in~\cite{CDGG} for Weinstein sectors (which can always be obtained from a Weinstein manifold with a chosen stop by removing a standard neighborhood of the stop); we highlight a difference in proof method for the reader.

		In~\cite{CDGG}, one tries to construct a Lagrangian $L'$ out of an initial Lagrangian $L$ such that (i) $L'$ is obtained by surgering $L$ along cocores, and (ii) $L'$ does not intersect the skeleton (hence is a zero object). In loc. cit., this strategy produces an $L'$ {\em immersed} in $M$, and the authors apply Floer theory for immersed objects by lifting to a Legendrian and passing to a Lagrangian cobordism for Legendrians. Computations utilizing an augmentation algebra show that one produces a twisted complex in the Fukaya category equivalent to $L$, constructed out of cocores.

		When proving results about the Lagrangian cobordism $\infty$-category, Floer-theoretic techniques usually fall short of proving structural statements of this sort. In particular, the augmentation algebra is not available in the present paper, as it is yet unclear how the differentials in the Floer chain complex (and other counts of holomorphic disks) relate to the topology of the space of Lagrangian cobordisms.

		The only step in which Floer-theoretic tools are crucially used in~\cite{CDGG}, and hence the main obstruction to extending their strategy, is the { immersed} nature of the Lagrangian surgeries: If a cocore disk intersects $L$ in more points than along the skeleton, we do not know how to construct a Lagrangian cobordism via surgeries while staying within the embedded setting.
\end{remark}

		\subsubsection{Proving the point acts.}
		The proof of Theorem~\ref{theorem. pt action} is different, and utilizes some simple observations about the theory of cobordisms. To give the reader an idea of what the proof feels like: A proper proof in similar spirit would show that the sphere spectrum (modeled as an infinite loop space of framed cobordisms) admits a commutative ring structure in the $\infty$-category of spectra. In this paper, we do not give a proper proof, in that we only prove results at the level of homotopy categories---as a result, our arguments are quite elementary and only involve simpleton observations about isotopies in $\RR^N$. 

	\subsection{Acknowledgments}
		We thank Sheel Ganatra, John Pardon, and Vivek Shende for helpful conversations. We would also like to thank the Korea Institute for Advances Study and American Institute of Mathematics for hosting the conference and workshop that spurred this paper. I was also supported by the Isaac Newton Institute for Mathematics Sciences during the Homotopy Harnessing Higher Structures semester in Fall 2018.

\section{Preliminaries}
	
	With the exception of Lemma~\ref{lemma.attachment}, every fact below is old news (i.e., has proofs in previous papers).

	\subsection{Geometric set-up and notation}\label{section. geometric setup}
	
\begin{notation}
		We will often use the notation $E = F = \RR$. This is to disambiguate the many roles that $\RR$ plays for the Lagrangian cobordism $\infty$-category---$E$ is a stabilization direction, while $F$ is a time/propagation/morphism direction.
\end{notation}
		
\begin{notation}
		We will often write $E^N \subset T^*E^N$ to mean the zero section.
\end{notation}

		\subsubsection{Liouville manifolds and stops}
			We engage in some rapid-fire recollections; details can be found, for example, in~\cite{ganatra-pardon-shende,gps2}.

			A Liouville manifold $(M,\theta)$ is an exact symplectic manifold equivalent to a symplectization outside some compact set. We let $\del_\infty M$ denote the associated contact manifold, so that $M \cong X^o \cup_{\del_\infty M} \del_\infty M \times [1,\infty)$ with $X^o \subset X$ compact. A {\em stop} is any closed subset $\ff \subset \del_\infty M$. (Note that this definition of stop is more general than the definition originally given by Sylvan~\cite{sylvan-thesis}.) The subset $X^o$ is called a {\em Liouville domain}.

			We let $Z$ denote the Liouville vector field of $M$, defined by $(d\theta)(Z,-) = \theta.$

			If $M$ and $Y$ are Liouville manifolds, so is $M \times Y$. 
			
			Finally, we say something is true {\em at infinity} if it is true inside a conical collar---equivalently, outside a sufficiently large compact set of $M$.

\begin{remark}
		Throughout this work, we follow the usual convention that Liouville domains $M^o$ are assumed compact. Note that this compactness assumption is unnecessary in the works~\cite{nadler-tanaka, tanaka-surgery, tanaka-pairing, tanaka-exact}.
\end{remark}

		In this work, we deal with the special case that $Z$ is gradient-like for a Morse function; we will call such an $M$ {\em Weinstein}. (We caution that the definition of Weinstein tends to vary from work to work in the literature, even in works by the same author.) 
		
\begin{remark}
		This main difference between the Weinstein and the general Liouville case is that Weinstein manifolds, like ordinary manifolds, can be built out of (symplectic) handle attachments. As a result, Weinstein manifolds also have natural Lagrangians submanifolds given as cocores of the attaching handles of index $n$ (where $\dim_\RR M = 2n$). It is more difficult to find Lagrangians inside arbitrary Liouville manifolds.
\end{remark}

\begin{construction}[Products of stops]\label{construction.product stops}
				If $\ff \subset \del_\infty M$ and $\fg \subset \del_\infty Y$ are choices of stops, we endow $M \times Y$ with the stop $\ff \times \core(Y) \cup \core(M) \times \fg \cup \ff \times \fg \times \RR.$
	\end{construction}

\begin{remark}
				Consider Liouville {\em domains} (which are compact, with boundary) $M^o$ and $Y^o$, so $M^o \times Y^o$ is a manifold with corners. In particular, $\del M^o \times \del Y^o \subset \del(M^o\times Y^o)$ admits a collar neighborhood topologically, but one must choose a smoothing to render $\del(M^o \times Y^o)$ smooth; this smoothing then contains a smooth copy of a collar neighborhood $(\del M^o \times \del Y^o) \times \RR$. This explains the $\RR$ factor in the term $\ff \times \fg \times \RR \subset \del_\infty M \times \del_\infty Y \times \RR$ in Construction~\ref{construction.product stops}.
	\end{remark}

\begin{example}
			Here are the examples of interest to us.
\begin{itemize}
				\item Let $M$ be a Weinstein manifold, so that $Z$ is gradient-like for some Morse function $f$. Then $\core(M)$ is the union of the stable manifolds that ascend to $\crit(f)$. 
				\item Let $E = \RR$ and let $T^*E$ be equipped with $Z = 1/2(q{\del_q} + p{\del_p})$. We choose a stop $\fg = \{\pm \infty\} \in [-\infty,\infty]$. Or, if one thinks of the unit disk as the corresponding Liouville domain, the stop $\fg$ is given by the points $\{\pm 1\}$ in the boundary of the unit disk.
				\item Then $M \times T^*E$ has a stop given by $\core(M) \times \fg$. 
				\item The product $T^*E^N$ can be thought of as associated to a Liouville domain diffeomorphic to $D^{2N} \subset \CC^N$, whose stop is the real equator $S^{N-1} \subset \RR^N \cap D^{2n}$.
	\end{itemize}
	\end{example}

\begin{remark}
		Let $M$ be Weinstein, possibly with a stop $\ff$, and consider $T^*E$ with stop $\fg$ as above. As we will see, in this paper, the invariants we care about for $M$ are equivalent to the invariants we care about for $M \times T^*E$. The reader may benefit from a reminder that, when we prove something for a Weinstein manifold $M$ in this paper, we will have also proven it for all of its stabilizations $M \times T^*E^N$, $N \geq 0$.
	\end{remark}
	
	\subsubsection{Lagrangians}
	
	A Lagrangian $L \subset M$ is {\em eventually conical} if, outside some compact set, $L$ is closed under the Liouville flow.
	
	A {\em primitive} for $L$ is a smooth function $f: L \to \RR$ such that $df = \theta|_L$. We will always demand $L$ to be equipped with a primitive $f$ such that $f$ vanishes outside a compact subset; in particular, this guarantees that $L$ is eventually conical. Given any eventually conical Lagrangian, such an $f$ can be obtained after a deformation of the Lagrangian. 
	
	{\em Brane} is a potentially vague term in this paper, but the only ambiguity is in the geometric decorations we demand of a Lagrangian. The data of a brane always has an underlying eventually conical Lagrangian $L$, equipped with a primitive vanishing outside a compact set; the additional tangential data we demand may include a grading or a relative Pin structure. Our theorems are true as stated without specifying what kind of brane convention $\cB$ we work with, so long as the Polterovich surgery construction works as in the framework of~\cite{tanaka-surgery}. 
	
	\subsubsection{Gradings}\label{section. gradings}
	Though gradings are not required as part of our brane structures, we will give a standard recollection, just because graded Lagrangians provide insightful examples of tangential structures interacting with the stability of the stable $\infty$-category.
	
	A grading on an $n$-dimensional Lagrangian is a generalization of the notion of an unsigned slope. 
	
	For example, when $L \subset T^*\RR^n$ is a linear subspace, let $A_L \in U(n)$ be a choice of real basis for $L$; then $A_L$ is uniquely determined up to an $O(n)$ action, and in particular $\det(A_L)^2$ is an invariant of $L$ as a linear Lagrangian in $T^*\RR^n \cong \RR^{2n}$. A {\em grading} on $L$ is a choice of real number $\alpha_L$ such that $e^{i \alpha_L} = \det(A_L)^2$. 
	
	Of course, we must specify how we think of $T^*\RR^n$ as a complex vector space. With the usual convention that $\theta = \sum_i p_i dq_i, \omega = \sum dp_idq_i$, and that $\omega(-,J-)$ defines a real inner product, we have $J(\del p_i) = \del q_i$ and $J(\del q_i) = - \del p_i$. 
	
	In the non-linear case, it is convenient to choose a global trivialization of the complex determinant bundle of $TM$---for example, this requires that $c_1(TM)$ be two-torsion. A lift of the $\det^2$ function $L \to S^1$ to a function $\alpha: L \to \RR$ is a {\em grading} on $L$. The obstruction to this lift is the class in $H^1(L,\ZZ)$ classified by $L \to S^1$. 

	\subsection{Stable infinity-categories}\label{section. stable infty-categories}
	
	We briefly recall the basics of $\infty$-categories. We recommend~\cite{htt} for a full account; the appendix of~\cite{nadler-tanaka} also has a presentation streamlined for our purposes.
	
	Recall that an $\infty$-category is a simplicial set satisfying the weak Kan condition~\cite{htt}. This has also been called a quasi-category by Joyal~\cite{joyal}, and a weak Kan complex by Boardman and Vogt~\cite{boardman-vogt}. 
	
\begin{defn}[Lurie~\cite{higher-algebra}]
	An $\infty$-category $\cC$ is called {\em stable} if the following three properties hold: (i) $\cC$ has a zero object, (ii) every morphism admits a fiber and a cofiber, and (iii) the induced shift functor is an equivalence. 
	\end{defn}
	
\begin{remark}
	Let us elaborate on the conditions for the reader's convenience. 
	
	(i) A zero object is an object $0 \in \cC$ such that $0$ is both terminal and initial---that is, for any $L \in \cC$, the spaces $\hom_\cC(0,L)$ and $\hom_\cC(L,0)$ are contractible. In particular, any two zero objects are equivalent in $\cC$. 
	
	(ii) An edge $\Delta^1 \to \cC$ in a quasi-category is a morphism; the 0-simplices $d_1(\cC)$ and $d_0(\cC)$ are the source and target, respectively, so we will often say ``let $Q: L_0 \to L_1$ be a morphism'' where $Q$ denotes the edge, and $L_0=d_1(Q)$ and $L_1=d_0(Q)$. 
	
	A {\em fiber} for $Q$ is a pullback of the diagram $L_0 \xra{Q} L_1  \leftarrow 0$; this is a functor $\Delta^1 \times \Delta^1 \to Q$, which we informally depict as
		\eqnn
		\xymatrix{
			F \ar[r] \ar[d] & L_0 \ar[d]^Q \\
			0 \ar[r] & L_1
		}
		\eqnd
	such that the diagram is a limit diagram in $\cC$. Dually, a {\em cofiber} for $Q$ is a pushout of the diagram $0 \leftarrow L_0 \xra{Q} L_1$, which can be informally depicted as
		\eqnn
		\xymatrix{
			L_0 \ar[r]^{Q} \ar[d] & L_1 \ar[d] \\
			0 \ar[r] & C.
		}
		\eqnd
	It is common to refer to the objects $F$ and $C$ as the fiber and cofiber of $Q$, but we will refer the reader to Warning~\ref{warning. fiber sequences}. 
	
	(iii) Finally, given any object $L$, one can now take the cofiber of the map $L \to 0$, which we will refer to as $L[1]$ or as $\Sigma L$. The fiber of the map $0 \to L$ is called $L[-1]$ or $\Omega L$. By the universal properties of (co)limits and zero objects, the assignments $L \mapsto L[\pm 1]$ define endofunctors $\cC \to \cC$; we demand that these be equivalences, and in fact, it follows that these are mutually inverse.  
	\end{remark}
	
\begin{remark}
		It is often said that in the setting of $\infty$-categories, a limit or a colimit is ``really'' a homotopy limit or a homotopy colimit. This has concrete truth if your $\infty$-category is known to be an $\infty$-category associated to a model category, and the distinction between the two kinds of limits (homotopy limit, or just limit) is important in model categorical language. However, in an $\infty$-category, there is only one definition of limit, and this definition of limit automatically demands a universal property with respect to all higher coherences. For example, there is no unique notion of composition in most $\infty$-categories, so the inarticulability of ``strict'' limits is a feature of the formalism. 
\end{remark}

\begin{warning}[Fiber sequences are more than just sequences]\label{warning. fiber sequences}
			In an $\infty$-category, a fiber sequence is not just the data of morphisms $A \to B \to C$; informally, a fiber sequence must come equipped with a homotopy from ``the'' composite $A \to C$ to the zero map $A \to 0 \to C$.\footnote{Again, composition is not uniquely defined in an $\infty$-category, but is only determined up to contractible choice.} Regardless, to save space and to evoke intuition from the theory of triangulated categories, we will often write ``let $A \to B \to C$ be a fiber sequence'' with the choice of null-homotopy non-explicit in the notation. A more proper convention would be: ``Let the map $\Delta^1 \times \Delta^1 \to \cC$, denoted by
				\eqnn
					\xymatrix{
					A \ar[r] \ar[d] & B \ar[d] \\
					0 \ar[r] & C,
					}
				\eqnd
			be a fiber sequence,'' but paper is precious.
\end{warning}
		
\begin{remark}
		Any pretriangulated dg-category is a stable $\infty$-category; this can be seen simply by applying the dg-nerve construction of Lurie~\cite{higher-algebra}; likewise for pretriangulated $A_\infty$-categories, by applying the $A_\infty$ nerve of Faonte or of the present author (see~\cite{faonte, tanaka-pairing}). It is also straightforward to verify (though the octahedral axiom is tedious as usual) that any stable $\infty$-category has a homotopy category which is triangulated: The distinguished triangles are those that descend from (co)fiber sequences. (A proof is in~\cite{higher-algebra}.)
		
		A frequently invoked slogan is that a stable $\infty$-category is like a triangulated category, except one remembers higher homotopies rendering mapping cones functorial, and one also allows for linearities that are not integral, but spectral.
\end{remark}

\begin{notation}\label{notation. span in oo-category}
			Let $\cC$ be a stable $\infty$-category and fix a collection of objects $\{D\} \subset \ob \cC$. We let $\langle \{D\} \rangle \subset \cC$ denote the smallest full subcategory of $\cC$ which contains $0$ and contains $\{D\}$, and is closed under pullbacks and pushouts.
\end{notation}

		For the record, we state the following straightforward result:

\begin{prop}\label{prop. generation and cones}
			Let $\cC$ be a stable $\infty$-category and fix a collection of objects $\{D\} \subset \ob \cC$. Then $X$ is in the essential image of the inclusion $\langle \{D\} \rangle \to \cC$ if and only if there exists objects $X = X_0, \ldots,  X_n \simeq 0$ such that for each $i$, there exist integers $a_i$ and objects $D_i \in \{D\}$ fitting into a fiber sequence
				\eqnn
				D_i[a_i] \to X_{i+1} \to X_{i}.
				\eqnd
\end{prop}

\begin{remark}
The above fiber sequence is of course equivalent to the fiber sequence 
				\eqnn
				X_{i+1} \to X_{i} \to 
				D_i[a_i+1].
				\eqnd
\end{remark}

	\subsection{The infinity-category of Lagrangian cobordisms}
		Fix $N \geq 0$. Let us first sketch the definition of an $\infty$-category $\lag^N_\Lambda(M)$. 
		
		An object $L \in \ob \lag^N_\Lambda(M)$ is the data of an eventually conical Lagrangian $L \subset M \times T^*E^N$, together with a brane structure which we shall not make explicit in the notation.  Moreover, $L$ must satisfy the following condition: The complement of the conical ends of $L$ must be contained in a compact region of $M \times T^*E^N$ (this means that, on the core of $M \times T^*E^N$, $L$ has compact support), and $L$ must avoid the stop of $M \times T^*E^N$ at infinity.

\begin{defn}[Morphisms]\label{defn. non-characteristic cobordism}
		Given two branes $L_0, L_1 \subset M \times T^*E^N$, a morphism is a choice of eventually conical Lagrangian brane $Q \subset M \times T^*E^N \times T^*F$, satisfying the following: 
			\enum
			\item
			(Collaring.) There exists $t_0 \leq t_1 \in F$ so that $Q|_{t \leq t_0} = L_0 \times (-\infty,t_0] \subset (M \times T^*E^N) \times T^*F$, where $(-\infty,t_0] \subset T^*F$ is a subset of the zero section. Likewise, we require $Q|_{t \geq t_1} = L_1 \times [t_1,\infty)$. 
			\item ($\Lambda$-avoiding.) Let $\tau$ denote the cotangent coordinate of $T^*F$. When $\tau <<0$, we demand that $Q$ does not intersect $\Lambda$. Because $Q$ is eventually conical, this means that if $Q$ has any conical part where $\tau \to -\infty$, the region $E^N \times F \times \{\tau << 0\} \times \Lambda$ does not intersect $Q$. We say that such a $Q$ is $\Lambda$-avoiding, or $\Lambda$-non-characteristic.
			\enumd
		Such a $Q$ is a morphism from $L_0$ to $L_1$, and is called a Lagrangian cobordism from $L_0$ to $L_1$. 
\end{defn}
		
\begin{example}
		The identity morphism of $L$ to itself is given by $L \times F \subset (M \times T^*E^N) \times T^*F$. 
\end{example}
		
		Higher morphisms are defined as collared, eventually conical Lagrangians $Q \subset M \times T^*E^N \times T^*F^M$ that are also $\Lambda$-avoiding. The collection of all (higher) morphisms and objects fit together to form a simplicial set; we model $\lag_\Lambda(M)$ as a quasi-category in this way. For details, see~\cite{nadler-tanaka} and~\cite{tanaka-pairing}.
		
		Importantly, one has a simplicial set $\lag^N_\Lambda(M)$ for every $N \geq 0$ consisting of objects contained in $M \times T^*E^N$ and morphisms between them; there are moreover stabilization maps
			\eqnn
			\dd: \Lag^N_\Lambda(M) \to \Lag^{N+1}_\Lambda(M),
			\qquad
			L \mapsto L \times T^*_{0}E
			\eqnd
		given by taking the direct product of a brane (or a cobordism) with the cotangent fiber at the origin of $E$. We define $\lag_\Lambda(M)$ to be the sequential colimit under these stabilization functors (or, by identifying a $\Lag^N_\Lambda(M)$ with its image under $\dd$, the union). In particular, in $\lag_\Lambda(M)$, an object $L$ is equivalent to its stabilization $L \times T^*_{0}E^N$ for any $N$. 
		
\begin{notation}
		Given $L \in \ob\lag_\Lambda(M)$, modeled as $L \subset M \times T^*E^N$ for some $N$, we let $L^{\dd} = L \times T^*_0 E$ be the stabilization, and we let $L^{\dd n} = L \times T^*_0 E^n$ be the $n$-fold stabilization.
\end{notation}

\begin{remark}
		Note that the $\Lambda$-avoiding condition is only imposed on the part of a cobordism $Q$ with {\em negative} $\tau$ coordinate, with no restrictions on the positive bit. This asymmetry is what makes $\lag_\Lambda(M)$ not an $\infty$-groupoid.
\end{remark}
		
		The main theorem of~\cite{nadler-tanaka} is that $\lag_\Lambda(M)$ is a stable $\infty$-category.  A zero object is given by the empty Lagrangian), and every morphism admits a fiber and cofiber. The cofiber, also known as the mapping cone, is as follows:
		
\begin{construction}[The mapping cone $\cone(Q)$] \label{construction. cone}
		Let $Q: L_0 \to L_1$ be a morphism; we model it by $Q \subset M \times T^*E^N \times F$. Then we can construct two connected, eventually conical curves $\gamma_0 , \gamma_1 \subset T^*F$ such that
\begin{itemize}
			\item $\gamma_0 \subset T^*(-\infty, t_0]$, $\gamma_0 = (t_0-\epsilon, t_0]$ near $t_0$, and $\gamma$ is equal to a vertical curve $\{t = const\}$ for $\tau<<0$, for some $t < t_0$. Further, $\gamma_0$ is equipped with a primitive vanishing where $\tau <<0$ and near $t_0 \in F$.
			\item Likewise, $\gamma_1 \subset T^*[t_1,\infty)$ with $\gamma_1$ equal to the zero section near $t_1$, and equal to a vertical curve $\{t = const\}$ for $\tau<<0$ and for some $t > t_1$. We demand $\gamma_1$ is also equipped with a primitive vanishing where $\tau<<0$ and near $t_1$. 
	\end{itemize}
		Then the Lagrangian
			\eqnn
				\cone(Q) :=
				(\gamma_0 \times L_0) \bigcup Q|_{[t_0,t_1]} \bigcup (\gamma_1 \times L_1)
				\subset
				M \times T^*E^{N+1}
			\eqnd
		is a model for the (homotopy) kernel of $Q$. When we demand that every brane is equipped with a grading, the grading on $Q$, together with the zero grading on $\RR \subset T^*\RR \cong \overline{\CC}$, induces a grading on the kernel. Shifting the grading by $+1$ on the same underlying Lagrangian, we obtain the cokernel of $Q$. We hence have a fiber sequence $L_0 \xra{Q} L_1 \to \cone(Q)$, see~\cite{nadler-tanaka}. 
\end{construction}
		
%		The reader may also check that when $L$ is equipped with a grading, the shift cokernel of the zero map, denoted $L[1]$, can be modeled by the same underlying Lagrangian, but with a shift in grading.
		
\begin{example}
		Let $\gamma \subset T^*F$ be a connected, properly embedded smooth curve such that $\gamma \cap T^*(-\infty,t_0] = (-\infty,t_0]$ for some $t_0$, and such that $\gamma$ is equal to a vertical curve $\{t = const\}$ for $\tau >>0$ and for some $t > t_0$; we also assume $\gamma$ admits a compactly supported primitive. Then one can deform $L \times \gamma$ to be eventually conical, and the result is a morphism from $L$ to the empty Lagrangian: a zero morphism. 
		
		The kernel of this zero morphism is, by construction, linearly Hamiltonian isotopic to a grading shift of $L^{\dd}$; specifically, the composite $L \times T^*_0 E \to L \xra{\alpha-1/2} \RR$ defines the natural grading.  Equivalently, the equivalence $L \simeq L \times T^*_0 E$ in $\lag_\Lambda(M)$ is realized by a brane $L \times T^*_0 E$ with the grading $\alpha + 1/2$. In particular, the shift functor $L \mapsto L[1]$ is the same on objects as in the partially wrapped Fukaya category with graded branes. 
\end{example}
		
		Below are some useful results about this $\infty$-category. 

\begin{theorem}[\cite{tanaka-surgery}]\label{theorem. surgery cone}
		Let $L_0, L_1 \subset M \times T^*E^n$ be Lagrangian branes. Suppose that $L_0$ and $L_1$ intersect at a unique point, and the intersection is transverse. Then for every brane structure on $L_0$, there is a brane structure $\alpha$ on $L_1$, along with a fiber sequence
			\eqnn
				L_1^\alpha \to L_1^\alpha \sharp L_0 \to L_0
			\eqnd
		where the middle term is a Polterovich surgery.
\end{theorem}
		
		(That is, Polterovich surgery ---also known as Lagrangian surgery---induces mapping cone sequences.)

\begin{prop}[\cite{nadler-tanaka}]\label{prop. zero objects}
			Let $L \subset M \times T^*E^n$ be a brane such that, for some neighborhood $U$ of $\sk(M) \times E^n$, we have that $L \cap U = \emptyset$. Then $L$ is a zero object in $\lag(M)$.
\end{prop}
		
		(That is, objects are ``supported'' along $\Lambda$.)
		
		Finally, wrappings are not a necessary component of defining the $\infty$-category of Lagrangian cobordisms; but wrapping is a useful way of producing equivalences:

\begin{defn}
				A Hamiltonian $H: M \to \RR$ is called {\em eventually linear} if $Z(H) = H$ near infinity.
	\end{defn}

\begin{prop}
				Let $H(-,-): \RR \times M \to \RR$ be a time-dependent function on $M$ such that for every $t \in \RR$, $H(t,-)$ is eventually linear, and such that $H_t = 0$ for $|t|>>0$. Then for any exact, eventually conical Lagrangian $L$, one has a Lagrangian cobordism from $L$ to the flow of $L$ at time $t>>0$ as follows:
					\eqnn
						L \times \RR \to M \times T^*F,
						\qquad
						(x,t) \mapsto (\flow_t^{X_H}(x), t, H_t(\flow_t^{X_H}(x)))
					\eqnd
	\end{prop}
			
\begin{remark}
			As usual, the Hamiltonian vector field $X_H$ is defined so that $\omega(-,X_H) = dH$. 
	\end{remark}

\begin{remark}\label{remark. Ham isotopies are equivalences away from Lambda}
			If the value of $H_t(x)$ is very positive or very negative, the ``eventually linear'' assumption implies that $x$ is far from the core of $M$. Thus so long as $H$ vanishes in a neighborhood of the stops in $\del_\infty M$, this cobordism is $\Lambda$-avoiding, and in fact, an equivalence, because it is bounded away from $\Lambda$ at $dt \to \infty$ too. (See Section~2.6 of~\cite{tanaka-exact}.) For example, Figure~\ref{figure. inverse} exhibits the higher cobordism showing that a composition of a morphism $Q$, and its ``$dt \mapsto -dt$'' flip $\overline{Q}$, compose to be higher-cobordant to the identity cobordism. Note that when $Q$ is vertically bounded (i.e., has bounded $dt$ coordinate), or avoids $\Lambda$ in both the $+\infty dt$ and $-\infty dt$ directions, then $\overline Q$ is indeed a morphism in $\lag_\Lambda(M)$, and the depicted higher cobordism is indeed a homotopy in $\lag_\Lambda(M)$. 
	\end{remark}

        \begin{figure}[h]
        		\[
        			\xy
        			\xyimport(8,8)(0,0){\includegraphics[width=4in]{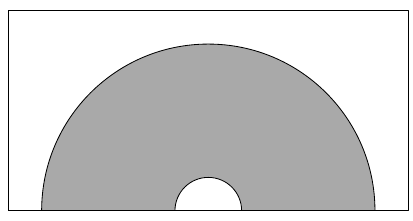}}
					,(2,1)*+{Q}
					,(6,1)*+{\overline{Q}}
					,(1,4)*+{L_0}
					,(7,4)*+{L_0}
					,(4,1)*+{L_1}
					,(4,8)*+{\id_{L_0}}					
        			\endxy
        		\]
        \caption{
        A homotopy showing $\overline Q \circ Q \sim \id_{L_0}$. A rotation in the other direction shows $Q \circ \overline{Q} \sim \id_{L_1}$. Thus, if $Q: L_0 \to L_1$ is vertically bounded, $Q$ is an equivalence.
        }
        \label{figure. inverse}
        \end{figure}

			Thus we have
			
\begin{prop}
				If two branes $L_0$ and $L_1$ in $\lag_\Lambda(M)$ are related by eventually linear Hamiltonian isotopies avoiding $\ff$, then $L_0$ and $L_1$ are equivalent objects. 
	\end{prop}

\begin{example}\label{example. vertical Hamiltonians from positive flows}
			Let $V$ be a contact vector field on $\del_\infty M$, meaning 
				$(\flow^V_t)^*\theta|_{\del_\infty M} = f_t \theta|_{\del_\infty M}$
			for some smooth, nowhere vanishing function $f_t: \del_\infty M \to \RR$. Choosing a Liouville flow coordinate $s$ near infinity, so that we have the identification
				$
					\lambda = e^s \lambda|_{\del_\infty M}
				$
			let us define $H = e^s \lambda|_{\del_\infty M}(V)$ near infinity, and extend $H$ arbitrarily into the interior of $M$. Then $H$ is eventually linear. 
			Our main source of contact vector fields will be those extended from Legendrian isotopies. 
	\end{example}

	\subsection{Attaching handles exactly}\label{section. handles}
	
		We refer the reader to \cite{gps2} for the notion of attaching an {\em exact} embedded Lagrangian $k$-handle; we recall some basics here. As before we will let $M^o \subset M$ denote a compact set given by the complement of some cylindrical end.
		
		If $L \subset M$ is a brane which is eventually conical with Legendrian boundary $A = \del_\infty L \subset \del_\infty M$, let $U \subset \del_\infty M$ be a Darboux-Weinstein neighborhood of $A$---then $U$ is contactomorphic to a neighborhood of the zero section of the jet bundle $J^1(A) \cong \RR_z \times T^*A$. The symplectization $\RR_{s>0} \times U \subset M$ is isomorphic to a neighborhood of the zero section of the cotangent bundle $T^*(\RR_{s>0} \times A)$. 
		
		The result of attaching an index $k$ exact embedded Lagrangian handle is a Lagrangian submanifold $\LL_1 \subset \RR_{s>0} \times U$, collared by $A$ at small $s$, and collared by a Legendrian $A'$ at larger $s$, where $A'$ is topologically the result of an index $k$ surgery of $A$. The union $\LL := (L \cap M^o) \bigcup_A \LL_1$ is an eventually conical Lagrangian brane inside $M$, and we call $\LL$ the result of attaching an {\em exact} embedded Lagrangian $k$-handle to $L$. 
		
		We emphasize that $\LL_1$ is contained entirely in $\RR_{s>0} \times U \subset M$. Moreover, if $L$ avoids $\Lambda$ and $\ff$ at infinity, then $\LL_1$ also avoids $\Lambda$ and $\ff \times \RR_s$ entirely. Finally, $\LL_1$ admits a primitive which vanishes at $A$ and at $A'$. (This is the reason for the adjective {\em exact} for this handle.)
		
\begin{remark}
		$\LL_1$ is topologically an elementary cobordism, and has been naturally called a Lagrangian cobordism from $A$ to $A'$ in the literature (see for example~\cite{dr}). Note that $\LL_1$ ``propagates'' in the conical direction of $\RR_s$, within the Liouville manifold $M$ itself; in particular it is not a Lagrangian cobordism in our sense, and does not represent a morphism in our $\infty$-category. It is simply a subset of a new object $\LL$.
\end{remark}
		
		In the partially wrapped setting, the Viterbo restriction functor, and the corestriction functor of completing Liouville subdomains, yield the following:
		
\begin{prop}\label{prop. handles dont change objects in fukaya}
		$L$ and $\LL$ are equivalent objects in the partially wrapped Fukaya category.
\end{prop}
		
\begin{proof}
		See Section~1.8 of \cite{gps2}.
\end{proof}
				
		The following is the analogue of Proposition~\ref{prop. handles dont change objects in fukaya} in the setting of Lagrangian cobordisms. We do not have, as far as we know, a Viterbo restriction functor for the $\infty$-category of Lagrangian cobordisms, so we present an alternate proof by constructing an explicit equivalence.

\begin{lemma}\label{lemma.attachment}
		There is an equivalence $\LL \to L$ in $\lag(M)$.
\end{lemma}
		
\begin{proof}
		We utilize the isomorphism $\RR_s \times J^1(A) \cong T^*(\RR_s \times A)$---this can be realized, for example, by
			\eqnn
			T^*(\RR_s \times A) \to \RR_s \times J^1(A),
			\qquad
			(s, a, \sigma, \alpha)
			\mapsto
			(\log s, a, \alpha/s, \sigma)
			\eqnd
		where $s \in \RR_s, a \in A, \sigma \in T^*\RR_s, \alpha \in T^*A$, and the last coordinate of $(\log s, a, \alpha/s, \sigma)$ is the 1-jet direction. (See Section~4 of~\cite{dr}, or Section~3 of~\cite{gps2} for a different map.) 
		
		Note that, crucially, one can choose a symplectomorphism as above so that a brane in $\RR_s \times J^1(A)$ is exact whenever it is exact in $T^*(\RR_s \times A)$ under the standard Liouville structures on both.
		
		Let $F = \RR_t$ (i.e., we will write an element of $F$ by $t$), and consider the Lagrangian $\LL_1 \times \RR_t \subset T^*(\RR_s \times A \times \RR_t)$, and for appropriate choices of real numbers $S, T >0$, let $\phi$ be a smooth embedding 
			$(0,S)_s \times (0,T)_t \to (0,S)_s \times (0,S)_t  \subset \RR_s \times \RR_t$ 
		as depicted in Figure~\ref{figure. phi}; informally, $\phi$ is a collared version of an orientation-reversing rotation of $\RR_s$ into $\RR_t$. We let $\Phi$ denote the induced symplectic embedding $T^*((0,S) \times A \times (0,T)) \to T^*(\RR_s \times A \times \RR_t)$. 
		
		Because $\LL_1$ is collared by $A$ for $s$ small, consider $Q_1 = \Phi(\LL_1 \times \RR_t)\bigcup A \times V$ where $V \subset (0,\infty)_s \times (0,\infty)_t$ is the complement of the image of $\Phi$. Via the inclusion $T^*(\RR_s \times A) \cong \RR_s \times J^1(A) \contains \RR_s \times U \subset M$, we consider $Q_1$ as a subset of $M \times T^*\RR_t$. Then the union $Q:= Q_1 \bigcup_{A \times \RR_t} \left((L \times \RR_t) \cap (M \times T^*\RR_t)^o\right)$  is a Lagrangian cobordism from $\LL$ to $L$. It is vertically bounded (i.e., the cotangent coordinate in the $\RR_t$ direction is bounded), so is an invertible morphism in $\lag_\Lambda(M)$. (This is the same reasoning as in Remark~\ref{remark. Ham isotopies are equivalences away from Lambda}.)
		
		Finally, note that one can as usual deform this Lagrangian in $M \times T^*F$ so its primitive vanishes near infinity, which (because the cobordism is vertically bounded) is where $|s| >>0$. On the other hand, we already know that $\LL_1$ was equipped with a vanishing primitive in this region, as did $\del_\infty L$, so we see that the resulting deformed Lagrangian is indeed still collared by $\LL$ and by $L$, as desired. 
\end{proof}

        \begin{figure}[h]
        		\[
        			\xy
        			\xyimport(8,8)(0,0){\includegraphics[width=4in]{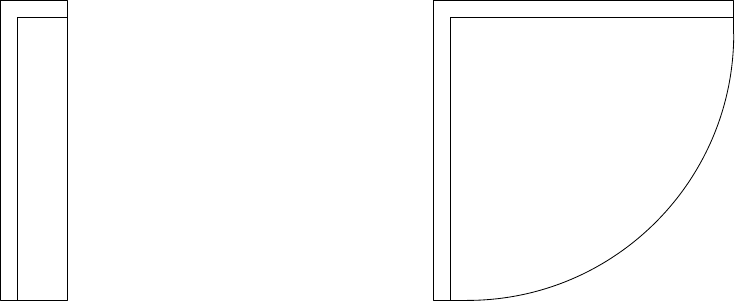}}
					,(-1,-0.5)*+{(s=0, t=0)}
					,(-1,8)*+{(s=S, t=0)}
					,(1.8,-0.5)*+{(s=0, t=T)}
					,(1.8,8)*+{(s=S, t=T)}
					,(6.3,5)*+{\phi( (0,S) \times (0,T) )}
        			\endxy
        		\]
        \caption{
        On the left, an open rectangle $(0,S)_s \times (0,T)_t$. On the right, the image of the open embedding $\phi$ taking the open rectangle to a quarter-disk-shaped region in $(0,S)_s \times (0,S)_t$. Note that the collarings near $s=S$ and $t=0$ are respected. 
        }
        \label{figure. phi}
        \end{figure}

\section{The proof of Theorem~\ref{theorem. cocores}}\label{section. proof generation}
	We follow the notation of~\cite{gps2}.

	Let $(M, \ff)$ be a Weinstein manifold with a stop $\ff = \fcrit \cup \fsubcrit \subset \del_\infty M$, where $\fcrit \subset \del_\infty M$ is a Legendrian.

	Let $Y = (\RR^2, \{\pm \infty\})$ denote $\RR^2$ with the radial Liouville structure, and two non-origin points chosen as stops, which we will call $\pm \infty$. This stopped Liouville manifold is equivalent to $T^*\RR$ with its usual cotangent Liouville structure. Up to isotopy, there is a unique Lagrangian linking disk in $Y$, which under this correspondence one can model as the cotangent fiber of $T^*\RR$.

	Then $M \times Y$, with the product stop, has the following property:

\begin{lemma}[\cite{gps2}.]\label{lemma. products of cocores}
		The product of a linking disk in $M$ and a linking disk in $Y$ is a linking disk in $M \times Y$, and the product of a cocore in $M$ with a linking disk in $Y$ is a linking disk in $M \times Y$. 
\end{lemma}
		
\begin{proof}
		See Section 7.2 of~\cite{gps2}.
\end{proof}

	Now we let $\ff$ be a stop and let $\Lambda = \core(M) \bigcup \bigcup_{t \in \RR} \flow^Z_t(\ff)$. 

	Note that because a linking disk in $Y$ is equivalent to a cotangent fiber of $T^*\RR$, the product of a Lagrangian $X$ in $M$ with a linking disk in $Y$ is the stabilization of $X$---in particular, the equivalence class of $X$ in $\lag_{\Lambda}(M)$ is unchanged.

	Given an object $X \in \lag_{\Lambda}(M)$, let $X^{\dd} = X \times \RR_y$ be the stabilization. There is an eventually linear, positive Hamiltonian isotopy that isotopes $\RR_y$ to a (non-compact) arc with $y$-coordinates strictly less than 0. An induced, eventually linear isotopy on $M \times Y$ moves $X^{\dd}$ to a brane whose $y$-coordinates are strictly less than zero. This mean that this brane is a zero object in $\lag(M)$. (Proposition~\ref{prop. zero objects}.)

	By genericity, one can arrange that this isotopy is modeled by passing through Lagrangian linking disks of $M \times Y$, one point at a time, and transversally. (Lemmas 2.2 and 2.3 of~\cite{gps2}.) These positive isotopies at infinity can be realized by eventually linear Hamiltonians by Example~\ref{example. vertical Hamiltonians from positive flows}. 

	To set notation, let us time-order the intersection points $x_1, \ldots, x_n$, so $x_i \in \flow_{t_i}(X^{\dd}) \cap D_i$ where $D_i$ is some Lagrangian linking disk, and $0 < t_1 < \ldots < t_n$. By Section~4 of~\cite{tanaka-surgery}, there exists a Polterovich surgery $D_i \sharp_{x_i} \flow_{t_i}(X^{\dd})$ equipped with a brane structure restricting to that of $X^{\dd}$ away from $x_i \cup D_i$. 

\begin{lemma}\label{lemma. surgery by linking disk}
		The surgery $D_i \sharp_{x_i} \flow_{t_i}(X^{\dd})$ is equivalent to $\flow_{t_i+\epsilon}(X^{\dd})$ in $\lag_\Lambda(M)$.
\end{lemma}
		
\begin{proof}
		For this proof only, let $\gamma_i$ denote the Reeb chord realizing the intersection $x_i$, and let $\flow_{t_i-\epsilon}(X^{\dd}) \sharp_{\gamma_i} D_i$ denote the result of attaching a {\em non-exact} embedded Lagrangian 1-handle as in \cite{gps2}. (Note that the object with subscript $\sharp_{\gamma_i}$ is a different submanifold than the object with subscript $\sharp_{x_i}$. Also, this non-exact handle attachment {\em does} result in an exact Lagrangian brane, but after potentially changing the brane structures of its constituents.) Let us attach an {\em exact} embedded $(n-1)$-handle to $\flow_{t_i-\epsilon}(X^{\dd}) \sharp_{\gamma_i} D_i$, and call the result $\XX_i$. By Remark~3.5 of loc. cit., $\XX_i$ is Hamiltonian isotopic to the surgery $D_i \sharp_{x_i} \flow_{t_i}(X^{\dd})$; this isotopy can be chosen to be compactly supported by well-chosen models for the surgery and the handle attachment. Further, by Lemma~\ref{lemma.attachment}, $\XX_i$ and $\flow_{t_i-\epsilon}(X^{\dd}) \sharp_{\gamma_i} D_i$ are equivalent by a vertically bounded morphism.
		
		Moreover, by Proposition 1.27 of~\cite{gps2}, we know that there is an eventually linear Hamiltonian isotopy between $\flow_{t_i+\epsilon}(X^{\dd})$ and $\flow_{t_i-\epsilon}(X^{\dd}) \sharp_{\gamma_i} D_i$. (Indeed, by calibrating our choices of wrapping, one may choose this isotopy to be compactly supported.) Thus we have the equivalences
			\eqn\label{eqn:main-chain-of-equivalences}
			\xymatrix{
			D_i \sharp_{x_i} \flow_{t_i}(X^{\dd})
			\ar[r]^-{\sim}
			&\XX_i
			\ar[rr]_-{Lemma~\ref{lemma.attachment}}^-{\sim}			
			&&\flow_{t_i-\epsilon}(X^{\dd}) \sharp_{\gamma_i} D_i
			\ar[r]^-{\sim}
			&\flow_{t_i+\epsilon}(X^{\dd}).
			}
			\eqnd
\end{proof}

	By Theorem~\ref{theorem. surgery cone}, each $D_i \sharp_{x_i} \flow_{t_i}(X^{\dd})$ admits a fiber sequence $D_i^{\alpha_i} \to D_i \sharp_{x_i} \flow_{t_i}(X^{\dd}) \to \flow_{t_i}(X^{\dd})$, where $D_i^{\alpha_i}$ is our notation for the disk $D_i$, equipped with a well-chosen brane structure $\alpha$. So by Lemma~\ref{lemma. surgery by linking disk}, we obtain fiber sequences
		\eqnn
			D_i^{\alpha_i} \to \flow_{t_i + \epsilon}(X^{\dd}) \to \flow_{t_i - \epsilon}(X^{\dd}),
			\qquad
			i = 1, \ldots, m.
		\eqnd

	Moreover, if the flow of $X^{\dd}$ does not pass through any linking disk in the time interval $[s,t]$,  then $\flow_s(X^{\dd})$ and $\flow_t(X^{\dd})$ are equivalent objects in $\lag(M)$ by Remark~\ref{remark. Ham isotopies are equivalences away from Lambda}.

	Because we know $\flow_0(X^{\dd}) = X^{\dd}$ and $\flow_{t_n + \epsilon}(X^{\dd})$ is a zero object, the theorem is proven by Proposition~\ref{prop. generation and cones}. One simply puts $X_i$ to be any of the equivalent objects in~\eqref{eqn:main-chain-of-equivalences}.

\section{Proof of Theorem~\ref{thm:geometric-version}}

\begin{notation}
In what follows, we let $M^\dd = M \times Y$.
\end{notation}

We fix times $t_i$ and objects $D_i^{\alpha_i}$ as in the proof of Theorem~\ref{theorem. cocores}. We note that the equivalence of Lemma~\ref{lemma.attachment} is realized by a vertically bounded morphism, while the other isotopies in the sequence~\eqref{eqn:main-chain-of-equivalences} can be chosen to be compactly supported. As such, the entire composition of equivalences in~\eqref{eqn:main-chain-of-equivalences} can be realized by a vertically bounded Lagrangian cobordism.

\begin{remark}
Let us briefly remark on the importance of our Hamiltonians being compactly supported. Fix a $t$-dependent Hamiltonian $H: M^\dd \times \RR_t \to \RR$. Then the resulting Lagrangian cobordism has the feature that at a given $(\flow_t(x),t) \in M^\dd \times \RR_t$, the cotangent fiber coordinate is proportional to $H_t(\flow_t(x))$. In particular, if $H_t$ ever approaches an arbitrary large value, we lose control over the collaring of the resulting cobordism. Maintaining this control is important, as in the statement of Theorem~\ref{thm:geometric-version}, we have strong restrictions on the collaring.
\end{remark}

As a result, each exact sequence
	\eqnn
	 D_i^{\alpha_i} \to X_i \to X_{i-1}
	 \eqnd
is realized by a Lagrangian cobordism $Q_i$ associated to surgery, possibly with vertically bounded Lagarngian cobordisms concatenated along the ends; in particular, the projection of $Q_i$ to $T^*F$ is collared by three rays---the incoming ray along the zero section $F \subset T^*F$ is collared by $D_i^{\alpha_i}$, while the outgoing ray is collared by $X_i$, and there is a vertical ray collared by $X_{i-1}$. Setting $A_1 = Q_1$ and proceeding by induction, we may glue $A_{i-1}$ along $X_{i-1}$ to obtain a new Lagrangian cobordism $A_i$. After the final step, we obtain a Lagrangian $A_n \subset M^\dd \times T^*F$ whose projection to $T^*F$ is equal to a finite number of rays outside of a compact subset of $T^*F$. Along these rays, $A_n$ is collared by $X_0 = X$, by the objects $D_i^{\alpha_i}$, and also by a zero object $X_n$.  This completes the proof.

\section{The symmetric monoidal structure}

\subsection{The definition}

We now define a functor $\tensor: \ho\Lag_{pt}(pt) \times  \ho\Lag_{pt}(pt) \to  \ho\Lag_{pt}(pt)$.

Recall that any $\infty$-category $\cC$ defines a category (in the usual sense, for instance, of MacLane~\cite{mac-lane}) $\ho\cC$ called its homotopy category. Its objects are the same as that of $\cC$, and the set of morphisms $\hom_{\ho\cC}(L,L')$ is given by $\pi_0\hom_\cC(L,L')$.

Fix two objects $L \subset T^*E^n$ and $L' \subset T^*E^{n'}$. As mentioned in \cite{tanaka-pairing, tanaka-exact}, the product of these need not be eventually conical. But assuming that their primitives vanish away from a compact set, one can always choose a deformation $L \tensor L'$ of $L \times L'$ which {\em is} eventually conical. (See for example \cite{tanaka-exact} or Section~6.2 of~\cite{gps2}.) We will assume we have made such a choice $L \tensor L'$ for every ordered pair $(L, L')$, and define $\tensor$ on objects as $(L,L') \mapsto L \tensor L'$. Note that any two such choices are Hamiltonian isotopic by an eventually linear Hamiltonian.

Fix two morphisms $P: L_0 \to L_1$ and $P': L'_0 \to L'_1$ in $\lag_{pt}(pt)$. We let $[P],[P']$ denote their morphisms in $\ho\Lag_{pt}(pt)$, and we define 
	\eqnn
	[P] \tensor [P'] := [(P \times \id_{L'_1}) \circ (\id_{L_0} \times P')] \in \pi_0 \hom_{\ho\Lag_{pt}(pt)}(L_0 \tensor L_0', L_1 \tensor L_1')
	\eqnd
	 to be the equivalence class of a chosen, eventually conical deformation of the composed cobordism $(P \times \id_{L'_1}) \circ (\id_{L_0} \times P')$. We choose this deformation so that the domain and codomain of this cobordism are indeed the objects $L_0 \tensor L_0'$ and $L_1 \tensor L_1'$, respectively.

\begin{remark}
Given $P$ and $P'$, note that $P \times P'$ most naturally lives over a rectangle, not over a line. (See Figure~\ref{figure.P-times-Q}.) The natural guess to extract a 1-morphism (i.e., something living over a line) would be to take a cobordism living along the diagonal of this rectangle, but the resulting submanifold may be singular. 

\begin{figure}
	\[
	\xymatrix{
	L_0 \times L'_1 \ar[rr]^{P \times \id_{L'_1}}
	&& L_1 \times L'_1 \\ 
	&P \times P'&\\
	L_0 \times L'_0 \ar[uu]^{\id_{L_0} \times P'} \ar[rr]_{P \times \id_{L'_0}}
	&& L_1 \times L'_0 \ar[uu]_{\id_{L_1} \times P'}
	}
	\]
\caption{The Lagrangian $P \times P' \subset T^*E^n \times T^*E^{n'} \times T^*F^2$. We have projected everything to $F^2$.}\label{figure.P-times-Q}
\end{figure}

What instead $P \times P'$ encodes is a higher cobordism (hence, in $\lag_{pt}(pt)$, a homotopy) between two natural morphisms:
	\[
	(P \times \id_{L'_1}) \circ (\id_{L_0} \times P')
	\qquad\text{and}\qquad
	(\id_{L_1} \times P') \circ (P \times \id_{L'_0})
	\]
which we can read off of the rectangle by considering the cobordisms collaring the boundary edges of the rectangle. Indeed, resolving the corners of the rectangle, we obtain a 2-morphism (a homotopy) showing that the two natural morphisms are homotopic in $\lag_{pt}(pt)$. Thus, either of these compositions is equally deserving of the title ``$P \tensor P'$.'' Indeed, it follows that $[(P \times \id_{L'_1}) \circ (\id_{L_0} \times P')] = [(\id_{L_1} \times P') \circ (P \times \id_{L'_0})]$ in $\ho\Lag$.

A similar argument shows that if $[P] = [Q]$ and $[P'] = [Q']$ in $\hom_{\ho\Lag}$, then $[P] \tensor [P'] = [Q] \tensor [Q']$. 
\end{remark}

\begin{remark}
The reader can anticipate a host of coherence questions that would arise by taking products of higher morphisms; and hence we hope we have conveyed a glimpse of the intricacies of defining the symmetric monoidal structure at the $\infty$-categorical level of $\lag$ (as opposed to the categorical level of $\ho\lag$). 
 \end{remark}

The following is now straightforward:

\begin{prop}
$\tensor: \ho\Lag_{pt}(pt) \times  \ho\Lag_{pt}(pt) \to  \ho\Lag_{pt}(pt)$ is a functor.
\end{prop}

\subsection{The swap map}\label{section.swap-map}
Any symmetric monoidal category comes equipped with a natural isomorphism
	\[
	s: L \tensor L' \cong L' \tensor L.
	\]
We define such a swap map for $\ho\lag_{pt}(pt)$.

\begin{remark}
Recall that given any time-dependent vector field $V_t$ on a smooth manifold $Q$, one has an induced Hamiltonian flow via the function
	\[
	H_t: T^*Q \to \RR,
	\qquad
	\alpha \mapsto \alpha(V_t).
	\]
The flow of $L_{H_t}$ and the flow of $V_t$ make the following diagram commute:
	\[
	\xymatrix{
	T^*Q \times \RR \ar[rr]^{\flow_{L_{H_t}}}  \ar[d]^{\pi \times \id_\RR} && T^*Q \ar[d]^{\pi} \\
	Q \times \RR \ar[rr]^{\flow_{V_t}}  && Q.
	}
	\]
Here, $\pi$ is the projection map from the cotangent bundle to the zero section. Note also that spin structures and gradings are preserved by Hamiltonian isotopies of a cotangent bundle induced by flows on the zero section.
\end{remark}

Let $L$ and $L'$ be branes in $T^*\RR^n$ and $T^*\RR^{n'}$, respectively. Stabilizing if necessary, we may assume $n=n'$. We define a morphism
	\[
	s: L \times L' \simeq L' \times L.
	\]
	
\enum
	\item 
By assumption, the reverse Liouville flow of $L \times L'$ is contained in some compact subset $K \subset \RR^n \times \RR^{n'}$.
Choose some vector $(v,v') \in \RR^n \times \RR^{n'}$ so that the translate
		\[
		K + (v,v')
		\]
	does not intersect the origin of $\RR^n \times \RR^{n'}$. Consider the object $(L \times L')^{\dd n} \subset (T^*\RR^n)^3$. Translation by $(v,v',0)$ defines a Hamiltonian isotopy of $(T^*\RR^n)^3$, which leaves the $\dd n$ coordinate unaffected. Perform this translation on $(L \times L')^{\dd n}$.
	\item Now consider the vector field
		\[
		\del_\theta : =
		\psi(r) \sum_{1 \leq i \leq n} 
		{\frac 
		{x_i \del_i - x_{n+i} \del_{n+i}}
		{x_i^2 + x_{n+i}^2}}
		\]
where $r: \RR^{3n} \to \RR$ is the distance from the origin, and $\psi(r)$ is a bump function which equals $1$ on $r(K)$, and zero outside a small neighborhood of $r(K)$. We rotate by 90 degrees via this vector field. This has the effect of fixing the ${\dd n}$ coordinate, but rotating the coordinates $x_i$ into $x_{n+i}$. Applying the resulting Hamiltonian vector field on $T^*\RR^{3n}$, one obtains a Lagrangian which is equal to $(\ov{L'} \times L)^{\dd n}$, translated. Note that $x_{n+i}$ is mapped to $-x_i$, hence the overline on $L'$.
	\item Now we likewise consider a vector field which rotates the $n'$ coordinate into the ${\dd n}$ coordinate, and apply a rotation of 180 degrees. The resulting Hamiltonian sends the Lagrangian from the previous step to one equalling a translate of $(L' \times L)^{\dd n}$.
	\item Translating back, we obtain the equivalence $(L \times L')^{\dd n} \to (L' \times L)^{\dd n}$.
\enumd

The only choices involved here were $(v,v')$, and the bump function $\psi$. The space of choices of $\psi$ is obviously contractible. Moreover, as $n, n' \to \infty$, the choice of vectors $(v,v') \in \RR^{n + n'} \setminus K$ is also contractible. Since the rotations and translations are determined completely by these two choices, we see that the map $s$ is well-defined up to contractible choice.

\begin{remark}\label{remark.s-s}
It is an easy exercise to perform $s \circ s$ and see that the resulting Lagrangian is equivalent to $L \times L'$, and that moreover $s \circ s$, being a result of Hamiltonian isotopies induced by vector fields on $\RR^N$, is homotopic to the identity cobordism.
\end{remark}

\begin{remark}
Naturality of $s$ also follows easily from the fact that $s$ is constructed out of eventually linear Hamiltonian isotopies.
\end{remark}

\subsection{The proof of Theorem~\ref{theorem. pt action}}

We first show that $\tensor$ and the swap map induce a symmetric monoidal structure
	\[
	\ho\lag_{pt}(pt) \times \ho\lag_{pt}(pt) \to \ho\lag_{pt}(pt)
	\]
on the homotopy category of $\lag_{pt}(pt)$; what remains is to construct associators and verify compatibilities.

The natural associators 
	\[
	\alpha: (L \tensor L') \tensor L'' \cong L \tensor (L' \tensor L'')
	\]
and natural isomorphism
	\[
	\rho: L \tensor 1 \cong L
	\]
are given by the obvious morphisms, using that the Cartesian product $L \times L' \times L''$ is associative (up to natural bijections of sets). The other natural isomorphism
	\[
	\lambda: 1 \tensor L\cong L,
	\]
is given by the swap map from Section~\ref{section.swap-map} showing $1 \tensor L \cong L \tensor 1$.

And in general, we take the braiding $s$ to be the swap map defined in Section~\ref{section.swap-map}. We must show that the following diagrams commute:
	\[
	\xymatrix{
	L \tensor 1 \ar[rr]^{s} \ar[dr]^{\rho} && 1 \tensor L \ar[dl]^{\lambda} \\
	&L &
	}
	\]
	\[
	\xymatrix{
	(L \tensor L') \tensor L'' \ar[r]^{s \tensor \id} \ar[d]^{\alpha} & (L' \tensor L) \tensor L'' \ar[d]^{\alpha} \\
	L \tensor (L' \tensor L'') \ar[d]^{s} & L' \tensor (L \tensor L'') \ar[d]^{\id \tensor s} \\
	(L' \tensor L'') \tensor L \ar[r]^{\alpha} 
	& L' \tensor (L'' \tensor L)
	}
	\]
	\[
	\xymatrix{
	 & L' \tensor L \ar[dr]^{s} \\
	 L \tensor L' \ar[ur]^{s} \ar[rr]^{\id}
	 && L \tensor L'.
	}
	\]
The bottom diagram commutes by Remark~\ref{remark.s-s}, and by definition, the top diagram is a special case of the bottom diagram. So only the middle diagram requires commentary. Without loss of generality, assume $L, L', L''$ are all submanifolds of $T^*\RR^n$ for the same $n$. Then the left vertical column is the swap map applied to $L^{\dd n}$ and $L' \times L''$. This is induced by the rotating vector field sending $x_i$ to $x_{2n+i}$ in $\RR^{4n}$.

The righthand column (together with the top horizontal arrow) is in contrast induced by two rotation maps on $\RR^{3n}$. First, the swap rotation sending $x_i$ to $x_{n+i}$, then the rotation map sending $x_{n+i}$ to $x_{2n+i}$. The end effect is a rotation map sending $x_i$ to $x_{2n+i}$, just as above. Moreover, $(L' \times L'' \times L)^{\dd n}$ is equivalent exactly to $(L' \times L'') \times L^{\dd n}$. Hence the middle diagram commutes as well.

So $\tensor$ lifts to a symmetric monoidal functor.

The bimodule action $\ho\lag_{pt}(pt) \times \ho\lag_{\Lambda}(M) \to \ho\lag_{\Lambda}(M)$,  $\ho\lag_{\Lambda}(M)\times \ho\lag_{pt}(pt) \to \ho\lag_{\Lambda}(M)$ is defined similarly, sending
	\eqnn
		(L,L') \mapsto L' \tensor L, \qquad (L',L) \mapsto L' \tensor L
	\eqnd
for $L \subset T^*E^n, L' \subset M \times T^*E^{n'}$. The compatibilities to be verified are analogous to the earlier bits of this section, so we do not repeat them.

So to finish the proof of the Theorem, we need only verify that mapping cones in each variable are sent to mapping cones. This is obvious, as given $P: L_0 \to L_1$, and an object $L'$, then $P \tensor \id_{L'} \simeq P \times L'$ as a manifold (up to Hamiltonian deformation), hence $[P \tensor \id_{L'}] = [P \times L']$, and the mapping cone of $P \times L'$ is geometrically constructed as $\cone(P) \times L'$. (See Construction~\ref{construction. cone}.) This completes the proof.

	\bibliographystyle{amsalpha}
	\bibliography{biblio}

\end{document}